\documentclass[a4paper,10pt]{amsart}
\usepackage[utf8]{inputenc}
\usepackage{amsxtra}
\usepackage{amsopn}
\usepackage{amsmath,amsthm,amssymb}
\usepackage{amscd}
\usepackage{mathtools}
\usepackage{amsfonts}
\usepackage{latexsym}
\usepackage{verbatim}
\usepackage{MnSymbol}
\usepackage{xcolor}
\usepackage{cases}
\usepackage{tikz-cd}
\usepackage{cases}

\let\c\overline
\theoremstyle{plain}
\newtheorem{theorem}{Theorem}[section]
\newtheorem*{theorem*}{Theorem}
\newtheorem*{theorem**}{Theorem \ref{thm-hyp}}

\newtheorem{lemma}[theorem]{Lemma}
\newtheorem{proposition}[theorem]{Proposition}

\newtheorem{remark}[theorem]{Remark}

\newtheorem*{question*}{Question}
\newtheorem*{mt*}{Main Theorem}
\newcommand\C{{\mathbb C}}

\newcommand\R{{\mathbb R}}
\newcommand\Z{{\mathbb Z}}

\newcommand{\de}[2]{\frac{\partial #1}{\partial #2}}

\newcommand{\del}{\partial}
\newcommand{\delbar}{\overline{\del}}
\renewcommand{\H}{\mathcal{H}}

\DeclareMathOperator{\vol}{Vol}

\DeclareMathOperator{\im}{im}

\let\c\overline
\let\phi\varphi
\title[Dolbeault harmonic $(1,1)$-forms on almost Hermitian $4$-manifolds]{On the dimension of Dolbeault harmonic $(1,1)$-forms on almost Hermitian $4$-manifolds}
\author{Riccardo Piovani}
\address{Dipartimento di Scienze Matematiche, Fisiche e Informatiche\\
Unit\`{a} di Matematica e Informatica\\
Universit\`{a} degli Studi di Parma\\
Parco Area delle Scienze 53/A \\
43124 Parma, Italy}
\email{riccardo.piovani@unipr.it}

\author{Adriano Tomassini}
\address{Dipartimento di Scienze Matematiche, Fisiche e Informatiche\\
Unit\`{a} di Matematica e Informatica\\
Universit\`{a} degli Studi di Parma\\
Parco Area delle Scienze 53/A \\
43124 Parma, Italy}
\email{adriano.tomassini@unipr.it}

\thanks{\newline 
The first author is partially supported by GNSAGA of INdAM.
The second author is partially supported by the Project PRIN 2017 ``Real and Complex Manifolds: Topology, Geometry and holomorphic dynamics'' and by GNSAGA of INdAM}

\begin{document}
\maketitle
%\tableofcontents

\begin{abstract} 
We prove that the dimension $h^{1,1}_{\delbar}$ of the space of Dolbeault harmonic $(1,1)$-forms is not necessarily always equal to $b^-$ on a compact almost complex $4$-manifold endowed with an almost Hermitian metric which is not locally conformally almost K\"ahler.
Indeed, we provide examples of non integrable, non locally conformally almost K\"ahler, almost Hermitian structures on compact $4$-manifolds with $h^{1,1}_{\delbar}=b^-+1$. This gives an answer to \cite[Question 3.3]{Ho} by Holt.
\end{abstract}

\section{Introduction}
Let $(M,J)$ be an almost complex manifold of real dimension $2n$. Given, on $(M,J)$, an almost Hermitian metric $g$, with fundamental form $\omega$, the triple $(M,J,\omega)$ will be called an almost Hermitian manifold. By the map $*:A^{p,q}\to A^{n-q,n-p}$, we denote the $\C$-linear extension of the real Hodge $*$ operator. The exterior derivative decomposes as 
\[
d=\mu+\del+\delbar+\c\mu,
\]
and we set
$\del^*:=-*\delbar*$ and $\delbar^*:=-*\del*$ to be the formal adjoint operators respectively of $\del$ and $\delbar$. Recall that
\begin{gather*}
\Delta_{\delbar}:=\delbar\delbar^*+\delbar^*\delbar
\end{gather*}
is the Dolbeault Laplacian, which is a formally self adjoint elliptic operator of order $2$. We set 
\[
\H^{p,q}_{\delbar}:=\ker\Delta_{\delbar}\cap A^{p,q},
\]
the space of Dolbeault harmonic $(p,q)$-forms.  If $M$ is compact, it is well known that the dimensions
\[
h^{p,q}_{\delbar}:=\dim_\C\H^{p,q}_{\delbar}
\]
are finite. Note that, a priori, $\del^*,\delbar^*,\Delta_{\delbar},\H^{p,q}_{\delbar},h^{p,q}_{\delbar}$ depend both on the almost complex structure $J$ and on the almost Hermitian metric $\omega$.

If $J$ is integrable, i.e., if $(M,J)$ is a complex manifold, then, assuming the compactness of $M$, by Hodge theory we know that the space of Dolbeault harmonic forms is isomorphic to the Dolbeault cohomology, i.e.,
\[
\H^{p,q}_{\delbar}\cong H^{p,q}_{\delbar}:=\frac{\ker\delbar\cap A^{p,q}}{\im \delbar\cap A^{p,q}}.
\]
Since the Dolbeault cohomology is a complex invariant, the numbers $h^{p,q}_{\delbar}$ do not depend on the choice of the metric in the integrable case.

Conversely, in the almost complex setting, Kodaira and Spencer asked the following question, which appeared as Problem 20 in Hirzebruch’s 1954 problem list \cite{Hi}: given a compact almost Hermitian manifold $(M,J,\omega)$, does $h^{p,q}_{\delbar}$ depend on the choice of the almost Hermitian metric $\omega$?
Recently in \cite{HZ} Holt and Zhang solved this problem, proving that $h^{p,q}_{\delbar}$ indeed depends on the choice of the metric. They provided an explicit example,  building a family of almost Hermitian structures on the Kodaira-Thurston manifold, which has real dimension $4$. They also proved that on every $4$-dimensional compact almost Hermitian manifold $(M,J,\omega)$, if the metric $\omega$ is almost K\"ahler, i.e., $d\omega=0$, then $h^{1,1}_{\delbar}=b^-+1$. Here $b^-$ is the dimension of the space of anti-self-dual harmonic forms, which is a topological invariant (see, e.g., \cite{DK}).

The study of the number $h^{1,1}_{\delbar}$ on compact almost Hermitian $4$-manifolds $(M,J,\omega)$ has been continued by Tardini and the second author in \cite{TT}.
Following their notation, we say that $\omega$ is \emph{strictly locally conformally almost K\"ahler} if
\[
d\omega=\theta\wedge\omega,
\]
and $\theta\in A^1$ is $d$-closed but non $d$-exact. Conversely, we say that $\omega$ is \emph{globally conformally almost K\"ahler}, if
\[
d\omega=\theta\wedge\omega,
\]
and $\theta\in A^1$ is $d$-exact. Indeed, if $\theta=dh$, then the metric $e^{-h}\omega$ is almost K\"ahler. Also note that, for any given almost Hermitian metric $\omega$, the $1$-form $\theta$ such that $d\omega=\theta\wedge\omega$ is uniquely determined by the Lefschetz isomorphism.
 We will also say that $\omega$ is \emph{locally conformally almost K\"ahler} if it is either strictly locally conformally almost K\"ahler or globally conformally almost K\"ahler.
 
Tardini and the second author proved that $h^{1,1}_{\delbar}=b^-$ on every compact almost complex 4-manifold with a strictly locally conformally almost K\"ahler metric. They also noted that $h^{1,1}_{\delbar}$ is a conformal invariant on almost Hermitian 4-manifolds, which implies that $h^{1,1}_{\delbar}=b^-+1$ on every compact almost complex 4-manifold with a globally conformally almost K\"ahler metric, by the previous result by Holt and Zhang in \cite{HZ}. 
Moreover, very recently, in \cite[Theorem 3.1]{Ho} Holt proved that $h^{1,1}_{\delbar}=b^-+1$ and $h^{1,1}_{\delbar}=b^-$ are the only two possible options on a compact almost Hermitian 4-manifold. In the remaining case when the almost Hermitian metric is not locally conformally almost K\"ahler, in general not much is known.
For other recent and related results concerning the study of the spaces of harmonic forms on compact almost Hermitian manifolds, see \cite{cattaneo-tardini-tomassini}, \cite{HZ2} and \cite{tardini-tomassini-dim6} for Dolbeault harmonic forms, and \cite{Ho}, \cite{piovani-tardini} and \cite{PT4} for Bott-Chern harmonic forms.

The study of $h^{1,1}_{\delbar}$ on compact almost Hermitian 4-manifolds is further motivated by what follows. In the integrable case, on a compact complex surface $(M,J)$, it is well known that the first Betti number $b^1$ is even if and only if $(M,J)$ admits a K\"ahler metric, and $b^1$ is even if and only if $h^{1,1}_{\delbar}=b^-+1$, see e.g., \cite{BPV}.  %The K\"ahler criterion using $b^1$ does not extend to non integrable almost complex structures. However, t
Thanks to the previously mentioned results, it is reasonable to expect $h^{1,1}_{\delbar}$ to detect almost K\"ahlerness. Motivated by this, Holt asked the following

\begin{question*}[{\cite[Question 3.3]{Ho}}]
On a compact almost Hermitian $4$-manifold, does the value of $h^{1,1}_{\delbar}$ gives a full description of whether an almost Hermitian metric is conformally almost K\"ahler? Specifically, in the case when the metric is not locally conformally almost K\"ahler, do we always have $h^{1,1}_{\delbar}=b^-$?
\end{question*}
In \cite{Ho} Holt also computed the value of $h^{1,1}_{\delbar}=b^-$ for a large family of almost Hermitian structures on the Kodaira-Thurston manifold and showed that his question is answered positively in that example.

%Let $M$ be a Hyperelliptic surface in the first isomorphism class as in Hasegawa, \cite[p. 755]{Ha}.%, corresponding to $\eta=\pi$ and $p=q=s=t=0$.
In this note, we 
%present examples of non-integrable almost complex structures on compact $4$-manifolds endowed with almost Hermitian metrics which are not locally conformally almost K\"ahler and show that $h^{1,1}_{\delbar}$ is not necessarily always equal to $b^-$. Indeed, we 
prove the following

\begin{theorem**}
Let $M$ be a Hyperelliptic surface in the first isomorphism class. On $M$, there exist two non integrable almost Hermitian structures $(J_1,\omega_1)$ and $(J_2,\omega_2)$ which are not locally conformally almost K\"ahler and which satisfy 
\[
h^{1,1}_{\delbar}(M,J_1,\omega_1)=2=b^-+1,\ \ \ h^{1,1}_{\delbar}(M,J_2,\omega_2)=1=b^-.
\]
\end{theorem**}

This gives a negative answer to \cite[Question 3.3]{Ho} by Holt.

The paper is organized as follows. Section \ref{sec-hyp} is devoted to the proof of Theorem \ref{thm-hyp}. The structure $(J_1,\omega_1)$ is obtained by deforming the standard K\"ahler structure $J_0$ on the Hyperelliptic surface $M$ into a non integrable almost complex structure $J_t$, and then by deforming the diagonal metric via a real parameter $\rho$ to obtain a metric $\omega_{t,\rho}$ which is not locally conformally almost K\"ahler and satisfies $h^{1,1}_{\delbar}=b^-+1$. The structure $(J_2,\omega_2)$ is obtained by deforming the diagonal metric, again via a real parameter, on a non integrable almost complex structure $J$ on $M$. The metric $\omega_2$ is not locally conformally almost K\"ahler and satisfies $h^{1,1}_{\delbar}=b^-$. In Section \ref{sec-nilm} we consider the $4$-dimensional nilmanifold $\mathcal{N}$ which does not admit any integrable almost complex structure. We choose an almost complex structure on $\mathcal{N}$, and the diagonal metric turns out to be not locally conformally almost K\"ahler. Again, we prove that $h^{1,1}_{\delbar}=b^-+1$, see Proposition \ref{prop-nilm}.

\medskip\medskip
\noindent{\em Acknowledgments.}
We are sincerely grateful to Tom Holt, for having shared with us his paper \cite{Ho} and to Lorenzo Sillari and Nicoletta Tardini, for interesting conversations on the subject of the paper. We also would like to thank the anonymous referee, for useful comments which improved the presentation of the results of the paper and for having pointed out an error in the argument of the proof of Theorem \ref{main}, which did not affect the result and which has been repaired in the present version.

\section{$h^{1,1}_{\delbar}$ of two almost Hermitian structures on a Hyperelliptic surface}\label{sec-hyp}

We start by recalling the description of the first isomorphism class of the Hyperelliptic surfaces as solvmanifolds.
Following Hasegawa, \cite[Section 3, pp. 754-755]{Ha}, let $G$ be the group $\C^2$ together with the operation
\[
(w^1,w^2)\cdot(z^1,z^2)=(w^1+e^{i\pi\frac{w^2+\c{w^2}}2}z^1,w^2+z^2),
\]
and let $\Gamma$ be the subgroup of $G$ given by $(\Z+i\Z)^2$. This corresponds to the hyperelliptic surface with $\eta=\pi$ and $p=q=s=t=0$ in the notation of Hasegawa. Let $M$ be the solvmanifold $\Gamma\backslash G$, and denote by $x^1,y^1,x^2,y^2$ the local coordinates of $M$ induced from $\C^2$, i.e., $z^1=x^1+iy^1$, $z^2=x^2+iy^2$. The vector fields
\begin{gather*}
e_1=\cos(\pi x^2)\de{}{x^1}+\sin(\pi x^2)\de{}{y^1},\\ e_2=-\sin(\pi x^2)\de{}{x^1}+\cos(\pi x^2)\de{}{y^1},\\ e_3=\de{}{x^2},\ e_4=\de{}{y^2}
\end{gather*}
are left invariant and form a basis of $TM$ at each point. The dual left invariant coframe is given by
\begin{gather*}
e^1=\cos(\pi x^2)dx^1+\sin(\pi x^2)dy^1,\\ e^2=-\sin(\pi x^2)dx^1+\cos(\pi x^2)dy^1,\\ e^3=dx^2,\ e^4=dy^2,
\end{gather*}
with structure equations
\begin{equation}\label{eq-struct-real}
de^1=-\pi e^{23},\ de^2=\pi e^{13},\ de^3=0,\ de^4=0.
\end{equation}
The de Rham cohomology of $M$ is computed by using left invariant forms, %see e.g. \cite{ADT}, 
therefore
\begin{equation*}
H_{dR}^1=\R<e^3,e^4>,\ \ \ H_{dR}^2=\R<e^{12},e^{34}>.
\end{equation*}

On $M$, consider the usual left invariant integrable almost complex structure $J_0$ given by
\begin{gather*}
\phi^1=e^1+ie^2,\\
\phi^2=e^3+ie^4
\end{gather*}
being a global coframe of $T^{1,0}(M,J_0)$.
The corresponding structure equations are
\begin{equation}\label{eq-struct-complex}
d\phi^1=i\frac\pi2(\phi^{12}+\phi^{1\c2}),\ \ \ d\phi^2=0.
\end{equation}
Endow $(M,J_0)$ with the K\"ahler metric given by the compatible symplectic form
\begin{equation*}
\omega=e^{12}+e^{34}=\frac{i}2(\phi^{1\c1}+\phi^{2\c2}).
\end{equation*}
Note that $b^-=1$ and $h^{1,1}_{\delbar}(M,J_0,\omega)=b^-+1=2$.

%In the remaining of the section, we present the following two non locally conformally almost K\"ahler structures on the Hyperelliptic surface $M$. First, we deform the standard K\"ahler structure on $M$ into a non integrable almost complex structure, and then we deform the diagonal metric with a real parameter to obtain an example where $h^{1,1}_{\delbar}=b^-+1$. Next, on the usual non integrable almost complex structure on $M$, we deform the diagonal metric, again with a real parameter, to obtain an example where $h^{1,1}_{\delbar}=b^-$.

%In the remaining of the section, we will prove the following
%\begin{proposition}\label{prop-hyp}
%Let $M$ be the Hyperelliptic surface. On $M$, there exist two non integrable almost Hermitian structures $(J_1,\omega_1)$ and $(J_2,\omega_2)$ which are not locally conformally almost K\"ahler and satisfy $h^{1,1}_{\delbar}(M,J_1,\omega_1)=2=b^-+1$ and $h^{1,1}_{\delbar}(M,J_2,\omega_2)=1=b^-$.
%\end{proposition}

%\subsection{Non-integrable deformation of the standard K\"ahler structure}

Let us consider the following left invariant non integrable almost complex deformation $(M,J_t)$ of the complex manifold $(M,J_0)$. 

\begin{lemma}\label{lemma-def}
For $t\in\C$, $|t|<1$, let $J_t$ be the almost complex structure on the Hyperelliptic surface $M$ defined by 
\begin{equation*}
\phi^1_t=\phi^1+t\phi^{\c1},\ \ \ \phi^2_t=\phi^2
\end{equation*}
being a coframe of $T^{1,0}(M,J_t)$. Then, the corresponding structure equations are
\begin{equation}\label{eq-struct-def}
d\phi^1_t=\frac{i\pi}{2(1-|t|^2)}\left((1+|t|^2)(\phi^{12}_t+\phi_t^{1\c2})+2t(\phi^{2\c1}_t-\phi^{\c1\c2}_t)\right),\ \ \ d\phi^2_t=0,
\end{equation}
and for $0<\rho<1$, $t\ne0$ the almost Hermitian metric
\begin{equation*}
\omega_{t,\rho}=\frac{i}2(\phi^{1\c1}_t+\phi^{2\c2}_t)+\frac\rho2(\phi^{1\c2}_t-\phi^{2\c1}_t),
\end{equation*}
is not locally conformally almost K\"ahler.
\end{lemma}
\begin{proof}
The structure equations \eqref{eq-struct-def} of $(M,J_t)$ derive from the structure equations \eqref{eq-struct-complex} of the complex manifold $(M,J_0)$ by elementary computations.\newline
Note that the diagonal metric
\begin{equation*}
\omega_t=\frac{i}2(\phi^{1\c1}_t+\phi^{2\c2}_t)
\end{equation*}
is almost K\"ahler. Conversely, let us show that $\omega_{t,\rho}$ is not locally conformally almost K\"ahler for $0<\rho<1$ and $t\ne0$.
In the following, it will be convenient to set
\[
\alpha=1+|t|^2-2t\in\C,\ \ \ \gamma=1-|t|^2\in\R.
\]
We compute
\[
d\omega_{t,\rho}=\frac{i\pi\rho}{4\gamma}\left(\c\alpha\phi^{12\c2}_t-\alpha\phi^{2\c1\c2}_t\right)=\theta_{t,\rho}\wedge\omega_{t,\rho}
\]
with
\[
\theta_{t,\rho}=\frac{\pi\rho}{2(1-\rho^2)\gamma}\left(\c\alpha(\phi^1_t+i\rho\phi^2_t)+\alpha(\phi^{\c1}_t-i\rho\phi^{\c2}_t)\right).
\]
For $0<\rho<1$ and $t\ne0$, the $1$-form $\theta_{t,\rho}$ is not closed, i.e., $d\theta_{t,\rho}\ne0$, therefore $\omega_{t,\rho}$ is not locally conformally almost K\"ahler.
\end{proof}

By Lemma \ref{lemma-def}, for $t\in\C$, $0\ne|t|<1$, and $0<\rho<1$, we have that $(J_t,\omega_{t,\rho})$ is a non integrable almost Hermitian structure on $M$ which is not locally conformally almost K\"ahler. %Let us compute the number $h^{1,1}_{\delbar}(M,J_t,\omega_{t,\rho})$.
We prove the following

\begin{theorem}
For $t\in\C$, $0\ne|t|<1$, and $0<\rho<1$, let $M$ be the Hyperelliptic surface endowed with the non integrable, non locally conformally almost K\"ahler structure $(J_t,\omega_{t,\rho})$ of Lemma \ref{lemma-def}. Then, 
\[
h^{1,1}_{\delbar}(M,J_t,\omega_{t,\rho})=2=b^-+1.
\]
\end{theorem}
\begin{proof}
The fundamental form $\omega_{t,\rho}$ can be also rewritten as
\[
\omega_{t,\rho}=\frac{i}2(\psi^{1\c1}+\psi^{2\c2}),
\]
where we set
\[
\psi^1=\phi^1_t+i\rho\phi^2_t,\ \ \ \psi^2=\sqrt{1-\rho^2}\phi^2_t
\]
as a unitary global coframe of $T^{1,0}(M,J_t)$.
Define the volume form $\vol$ such that
\begin{equation*}
\vol=\frac{\omega_{t,\rho}^2}2=\frac14\psi^{12\c1\c2}.
\end{equation*}
In the following computations, it will be convenient to set
\[
\alpha=1+|t|^2-2t\in\C,\ \ \ \beta=1+|t|^2\in\R,\ \ \ \gamma=1-|t|^2\in\R.
\]
Let $\eta\in A^{1,1}(M,J_t)$ be a left invariant $(1,1)$-form on $(M,J_t)$. We can write
\[
\eta=A\psi^{1\c1}+B\psi^{1\c2}+C\psi^{2\c1}+D\psi^{2\c2},
\]
where $A,B,C,D\in\C$.
Applying the Hodge $*$ operator we get
\[
*\eta=D\psi^{1\c1}-B\psi^{1\c2}-C\psi^{2\c1}+A\psi^{2\c2}.
\]
We compute
\begin{align*}
\delbar\eta&=\frac{\pi}{2\gamma}\left(-A\rho\alpha+B2it\sqrt{1-\rho^2}+Ci\beta\sqrt{1-\rho^2}\right)\phi^{2\c1\c2}_t,
\end{align*}
and
\begin{align*}
\del*\eta&=\frac{\pi}{2\gamma}\left(D\rho\c\alpha-Bi\beta\sqrt{1-\rho^2}-C2i\c{t}\sqrt{1-\rho^2}\right)\phi^{12\c2}_t.
\end{align*}
Therefore $\delbar\eta=\del*\eta=0$ iff
\begin{numcases}{}
\label{eq delbar}
-A\rho\alpha+B2it\sqrt{1-\rho^2}+Ci\beta\sqrt{1-\rho^2}=0,\\
\label{eq delbarstar}
D\rho\c\alpha-Bi\beta\sqrt{1-\rho^2}-C2i\c{t}\sqrt{1-\rho^2}=0.
\end{numcases}
We sum \eqref{eq delbar} multiplied by $\beta$ and \eqref{eq delbarstar} multiplied by $2t$. We derive
\begin{equation}\label{eq c}
C=\frac{i\rho}{\gamma^2\sqrt{1-\rho^2}}\left(-A\beta\alpha+D2t\c\alpha\right).
\end{equation}
We sum \eqref{eq delbar} multiplied by $2\c{t}$ and \eqref{eq delbarstar} multiplied by $\beta$. We derive
\begin{equation}\label{eq b}
B=\frac{-i\rho}{\gamma^2\sqrt{1-\rho^2}}\left(-A2\c{t}\alpha+D\beta\c\alpha\right).
\end{equation}
Substituting \eqref{eq c} and \eqref{eq b} into the system given by \eqref{eq delbar} and \eqref{eq delbarstar}, we find
\begin{equation}\label{eq-ident}
\begin{cases}
A\alpha\rho(\gamma^2+4|t|^2-\beta^2)=0,\\
D\c\alpha\rho(\gamma^2+4|t|^2-\beta^2)=0.
\end{cases}
\end{equation}
Finally, note that \eqref{eq-ident} is an identity, since by definition
\[
\gamma^2+4|t|^2-\beta^2=0.
\]
This implies that $\eta\in\H^{1,1}_{\delbar}(M,J_t,\omega_{t,\rho})$ for every $A,D\in\C$ and for any $C,B\in\C$ given by equations \eqref{eq c} and \eqref{eq b}. It follows that $h^{1,1}_{\delbar}(M,J_t,\omega_{t,\rho})\ge2=b^-+1$. By \cite[Theorem 3.1]{Ho}, $h^{1,1}_{\delbar}$ can be equal only to $b^-$ or $b^-+1$, therefore we conclude that $h^{1,1}_{\delbar}(M,J_t,\omega_{t,\rho})=2=b^-+1$.
\end{proof}

In the remainder of this section, we will construct another non integrable almost complex structure $J$ on $M$ admitting an almost Hermitian metric $\omega$ which is Gauduchon (that is, in real dimension $4$, $\del\delbar\omega=0$ or, equivalently, $dJd\omega=0$) but not locally conformally almost K\"ahler. This structure extends the almost Hermitian structure on $M$ considered in \cite[Section 6]{PT4}.

\begin{lemma}\label{lemma-gau}
On the Hyperelliptic surface $M$, define the almost complex structure $J$ given by
\begin{gather*}
\phi^1=e^1+ie^3,\ \ \ 
\phi^2=e^2+ie^4
\end{gather*}
being a coframe of $T^{1,0}(M,J)$. Then, the corresponding structure equations are
\begin{equation}\label{eq-struct-gau}
d\phi^1=i\frac\pi4(-\phi^{12}-\phi^{1\c2}-\phi^{2\c1}+\phi^{\c1\c2}),\ \ \ d\phi^2=i\frac\pi2\phi^{1\c1},
\end{equation}
and for $0<\rho<1$ the almost Hermitian metric
\begin{equation*}
\omega_\rho=e^{13}+e^{24}+\rho e^{12}+\rho e^{34}=\frac{i}2(\phi^{1\c1}+\phi^{2\c2})+\frac\rho2(\phi^{1\c2}-\phi^{2\c1}),
\end{equation*}
is Gauduchon but not locally conformally almost K\"ahler.
\end{lemma}
\begin{proof}
The structure equations \eqref{eq-struct-gau} of $(M,J)$ derive from the real structure equations \eqref{eq-struct-real} of $M$ by elementary computations.\newline
Let us show that $\omega_{\rho}$ is not locally conformally almost K\"ahler for $0<\rho<1$.
We compute
\[
d\omega_\rho=\pi e^{134}=\theta_\rho\wedge\omega_\rho,
\]
with
\[
\theta_\rho=-\frac{\pi\rho}{1-\rho^2}e^1+\frac{\pi}{1-\rho^2}e^4.
\]
For $0<\rho<1$ the $1$-form $\theta_\rho$ is not closed, i.e., $d\theta_\rho\ne0$, therefore $\omega_{\rho}$ is not locally conformally almost K\"ahler. 
For $\rho=0$, note that $\theta_0$ is closed but non exact, thus $\omega_0$ is strictly locally conformally almost K\"ahler.\newline
Finally, since $\omega_\rho$ is left invariant, it is Gauduchon. However, let us check it explicitly. Indeed, for $0\le\rho<1$, we have
\begin{equation*}
dJd\omega_\rho=dJ\pi e^{134}=-\pi de^{312}=0.\qedhere
\end{equation*}
\end{proof}

By Lemma \ref{lemma-gau}, for $0<\rho<1$, we have that $(J,\omega_{\rho})$ is a non integrable almost Hermitian structure on $M$ which is not locally conformally almost K\"ahler. Let us compute the number $h^{1,1}_{\delbar}(M,J,\omega_{\rho})$. We will need the following result (see \cite[Theorem 3.6]{P}). Recall that $d^c=J^{-1}dJ=i(\delbar-\del+\mu-\c\mu)$.
\begin{theorem}\label{teorema-piovani}
Let $G$ be a $4$-dimensional Lie group and let $\Gamma$ be a discrete subgroup such that $M=\Gamma\backslash G$ is compact. Assume that $(J,\omega)$ is a left invariant almost Hermitian structure on $G$. Then $h^{1,1}_{\delbar}(M,J,\omega)=b^-+1$ if and only if there exists a left invariant anti-self-dual $(1,1)$-form $\gamma\in A^{1,1}(G,J)$ satisfying 
$$
id^c\gamma =d \omega.
$$
Otherwise, $h^{1,1}_{\delbar}(M,J,\omega)=b^-$.
\end{theorem}

%\subsection{Non-integrable standard almost complex structure}
%\begin{remark}\label{rem-func-hyp}
%As $M$ is a quotient of $\C^2$ by the action of $\Gamma$, we can see a smooth function $f:M\to\C$ as a smooth function $f:\C^2\to\C$ which satisfies the following invariance property
%\[
%f(z^1,z^2)=f(w^1+e^{i\pi\frac{w^2+\c{w^2}}2}z^1,w^2+z^2)
%\]
%for every $(w^1,w^2)\in(\Z+i\Z)^2$.
%\end{remark}
We are now ready to prove the following
\begin{theorem}\label{main}
For $0<\rho<1$, let $M$ be the Hyperelliptic surface endowed with the non integrable, non locally conformally almost K\"ahler structure $(J,\omega_\rho)$ of Lemma \ref{lemma-gau}. Then 
\[
h^{1,1}_{\delbar}(M,J,\omega_\rho)=b^-=1.
\]
\end{theorem}
\begin{proof}
The fundamental form $\omega_\rho$ can also be rewritten as
\[
\omega_\rho=\frac{i}2(\psi^{1\c1}+\psi^{2\c2}),
\]
where we set
\[
\psi^1=\phi^1+i\rho\phi^2,\ \ \ \psi^2=\sqrt{1-\rho^2}\phi^2
\]
as a unitary coframe of $T^{1,0}(M,J)$.
We have
\begin{align*}
&\delbar\psi^1=-i\frac\pi4\phi^{1\c2}-i\frac\pi4\phi^{2\c1}-\rho\frac\pi2\phi^{1\c1},\\
&\del\psi^1=-i\frac\pi4\phi^{12},\\
&\delbar\psi^2=i\frac\pi2\sqrt{1-\rho^2}\phi^{1\c1},\\
&\del\psi^2=0.
\end{align*}
Define the volume form $\vol$ such that
\begin{equation*}
\vol=\frac{\omega_\rho^2}2=\frac14\psi^{12\c1\c2}.
\end{equation*}

In order to compute $h^{1,1}_{\delbar}(M,J,\omega_\rho)$, we will apply \cite[Theorem 3.6]{P}, see Theorem \ref{teorema-piovani} above, and therefore it suffices to show that there are no left invariant anti-self-dual $(1,1)$-forms $\gamma$ satisfying
$$
id^c\gamma=d\omega_\rho.
$$ 
For bidegree reasons, we have $id^c\gamma=\del\gamma-\delbar\gamma$.
Any left invariant anti-self-dual $(1,1)$-form $\gamma$ can be written as
\[
\gamma=B\psi^{1\c1}+C\psi^{1\c2}+D\psi^{2\c1}-B\psi^{2\c2},
\]
with $B,C,D\in\C$. Note that $*\gamma=-\gamma$, i.e., $\gamma$ is anti-self-dual.
We have
\begin{align*}
&\delbar\psi^{1\c1}=i\frac\pi2\rho^2\phi^{1\c1\c2},\\
&\delbar\psi^{1\c2}=-\frac\pi2\rho\sqrt{1-\rho^2}\phi^{1\c1\c2}-i\frac\pi4\sqrt{1-\rho^2}\phi^{2\c1\c2},\\
&\delbar\psi^{2\c1}=\frac\pi2\rho\sqrt{1-\rho^2}\phi^{1\c1\c2}-i\frac\pi4\sqrt{1-\rho^2}\phi^{2\c1\c2},\\
&\delbar\psi^{2\c2}=i\frac\pi2(1-\rho^2)\phi^{1\c1\c2}.
\end{align*}
Moreover, $\del\psi^{i\c{j}}=-\c{\delbar\psi^{j\c{i}}}$ for $i,j\in\{1,2\}$, therefore we obtain
\begin{align*}
\delbar\gamma&=\frac\pi2\Big(i(2\rho^2-1)B-\rho\sqrt{1-\rho^2}C+\rho\sqrt{1-\rho^2}D\Big)\phi^{1\c1\c2}+\\
&-i\frac\pi4\sqrt{1-\rho^2}\Big(C+D\Big)\phi^{2\c1\c2}
\end{align*}
and
\begin{align*}
\del\gamma&=\frac\pi2\Big(i(2\rho^2-1)B-\rho\sqrt{1-\rho^2}C+\rho\sqrt{1-\rho^2}D\Big)\phi^{12\c1}+\\
&-i\frac\pi4\sqrt{1-\rho^2}\Big(C+D\Big)\phi^{12\c2}.
\end{align*}
We have already computed 
\begin{align*}
d\omega_\rho&=\pi e^{134}\\
&=\pi\frac{i}2\phi^{1\c1}\wedge\frac1{2i}(\phi^2-\phi^{\c2})\\
&=-\frac\pi4(\phi^{12\c1}+\phi^{1\c1\c2}),
\end{align*}
%&=-\frac{i}2\pi\phi^{1\c1}\wedge\frac1{2i}(\phi^2-\phi^{\c2})
therefore $id^c\gamma=d\omega_\rho$ if and only if
\begin{equation*}
\begin{cases}
\delbar\gamma=\frac\pi4\phi^{1\c1\c2},\\
\del\gamma=-\frac\pi4\phi^{12\c1},
\end{cases}
\end{equation*}
if and only if
\begin{equation*}
\begin{cases}
C+D=0,\\
i(2\rho^2-1)B-\rho\sqrt{1-\rho^2}C+\rho\sqrt{1-\rho^2}D=\frac12,\\
i(2\rho^2-1)B-\rho\sqrt{1-\rho^2}C+\rho\sqrt{1-\rho^2}D=-\frac12,
\end{cases}
\end{equation*}
which is a contradiction. It follows that there are no left invariant anti-self-dual $(1,1)$-forms $\gamma$ satisfying $id^c\gamma=d\omega_\rho$ and so $h^{1,1}_{\delbar}(M,J,\omega_\rho)=b^-=1$ by Theorem \ref{teorema-piovani}.
\end{proof}

Summing up, we have just proved the following
\begin{theorem}\label{thm-hyp}
Let $M$ be a Hyperelliptic surface in the first isomorphism class. On $M$, there exist two non integrable almost Hermitian structures $(J_1,\omega_1)$ and $(J_2,\omega_2)$ which are not locally conformally almost K\"ahler and which satisfy 
\[
h^{1,1}_{\delbar}(M,J_1,\omega_1)=2=b^-+1,\ \ \ h^{1,1}_{\delbar}(M,J_2,\omega_2)=1=b^-.
\]
\end{theorem}

\section{$h^{1,1}_{\delbar}$ on the nilmanifold $\mathcal{N}$}\label{sec-nilm}

We start by recalling the description of the nilmanifold $\mathcal{N}$.
Following Hasegawa, \cite[Section 4, p. 758]{Ha}, let $G$ be the group $\R^4$ together with the operation
\[
(a,b,c,d)\cdot(x,y,z,t)=(x+a,y+b,z+ay+c,t+\frac12a^2y+az+d).
\]
Then the following $1$-forms
\begin{align*}
&e^1=dx,\\
&e^2=dy,\\
&e^3=dz-xdy,\\
&e^4=dt+\frac12x^2dy-xdz,
\end{align*}
on $G$ form a basis for $A^1(G)$, and they are left invariant. Let $\Gamma$ be the subgroup of $G$ given by $(2\Z)^4$ and call $\mathcal{N}$ the compact quotient $\Gamma\backslash G$. It follows that $e^1,e^2,e^3,e^4$ are still $1$-forms on $\mathcal{N}$ and form a basis of $A^1(\mathcal{N})$. The structure equations for $\mathcal{N}$ are
\[
de^1=0,\ \ \ de^2=0,\ \ \ de^3=-e^{12},\ \ \ de^4=-e^{13}.
\]
\begin{remark}
As $\mathcal{N}$ is a quotient of $\R^4$ by the action of $\Gamma$, we can see a smooth function $f:\mathcal{N}\to\R$ as a smooth function $f:\R^4\to\R$ which satisfies the following invariance property
\[
f(x,y,z,t)=f(x+2\alpha,y+2\beta,z+2\alpha y+2\gamma,t+2\alpha^2y+2\alpha z+2\delta)
\]
for every $(\alpha,\beta,\gamma,\delta)\in\Z^4$.
\end{remark}

We endow $\mathcal{N}$ with the non integrable almost complex structure $J$ defined by taking
\[
\phi^1=e^1+ie^2,\ \ \ \phi^2=e^3+ie^4
\]
as a basis of $A^{1,0}(\mathcal{N},J)$. The structure equations are
\[
d\phi^1=0,\ \ \ d\phi^2=-\frac{i}4(\phi^{12}+2\phi^{1\c1}+\phi^{1\c2}-\phi^{2\c1}+\phi^{\c1\c2}).
\]
We consider, on $(\mathcal{N},J)$, the almost Hermitian metric given by the fundamental form
\[
\omega=e^{12}+e^{34}=\frac{i}2(\phi^{1\c1}+\phi^{2\c2}).
\]
Take
\[
\vol=\frac{\omega^2}2=e^{1234}=\frac14\phi^{12\c1\c2}
\]
as the volume form.
The de Rham cohomology can be computed by looking only at the invariant forms, therefore
\[
H^2_{dR}=\R<e^{14},e^{23}>
\]
and $b^-=1$.
Since
\[
d\omega=-e^{124}=\theta\wedge\omega
\]
with $\theta=-e^4$, which is not closed, it follows that $\omega$ is not locally conformally almost K\"ahler. We are going to show that $h^{1,1}_{\delbar}(\mathcal{N},J,\omega)=b^-+1$.

Let $\eta\in A^{1,1}(\mathcal{N},J)$ be a left invariant $(1,1)$-form. We can write
\[
\eta=A\phi^{1\c1}+B\phi^{1\c2}+C\phi^{2\c1}+D\phi^{2\c2},
\]
where $A,B,C,D\in\C$.
Applying the Hodge $*$ operator we get
\[
*\eta=D\phi^{1\c1}-B\phi^{1\c2}-C\phi^{2\c1}+A\phi^{2\c2}.
\]
Recall that $\eta\in\H^{1,1}_{\delbar}$ iff $\delbar\eta=\del*\eta=0$. We compute
\[
\delbar\eta=\frac{i}4(-B+C-2D)\phi^{1\c1\c2},
\]
and
\[
\del*\eta=\frac{i}4(-B+C-2A)\phi^{12\c1}.
\]
It follows that the dimension of the space of invariant Dolbeault harmonic $(1,1)$-forms is $2=b^-+1$. By \cite[Theorem 3.1]{Ho}, $h^{1,1}_{\delbar}$ can be equal only to $b^-$ or $b^-+1$, therefore we conclude that $h^{1,1}_{\delbar}(\mathcal{N},J,\omega)=b^-+1$. Explicitly,
\[
\H^{1,1}_{\delbar}(\mathcal{N},J,\omega)=\C<2\phi^{1\c2}-\phi^{1\c1}-\phi^{2\c2},2\phi^{2\c1}+\phi^{1\c1}+\phi^{2\c2}>.
\]

We have just proved the following
\begin{proposition}\label{prop-nilm}
Let $\mathcal{N}$ be the 4-manifold previously defined. On $\mathcal{N}$, there exists a non integrable almost Hermitian structure $(J,\omega)$ which is not locally conformally almost K\"ahler and which satisfies $h^{1,1}_{\delbar}(\mathcal{N},J,\omega)=2=b^-+1$.
\end{proposition}

\end{document}